\documentclass{article}
\usepackage{amsmath,amssymb,enumerate, fullpage}
\usepackage{tikz,float}

\newcommand{\tmstrong}[1]{\textbf{#1}}
\newenvironment{enumerateroman}{\begin{enumerate}[i.] }{\end{enumerate}}
\newenvironment{itemizeminus}{\begin{itemize} }{\end{itemize}}
\newenvironment{proof}{\noindent\textbf{Proof\ }}{\hspace*{\fill}$\Box$\medskip}
\newtheorem{proposition}{Proposition}

\begin{document}

\title{Non-simultaneous match-stick geometry}\author{}\maketitle

{\tmstrong{\title{Pfannerer Stephan \footnote[1]{e1026392@student.tuwien.ac.at; Vienna University of Technology}, Schram
Philippe \footnote[2]{pschram@ulb.ac.be; Universit\'{e} libre de Bruxelles}}\author{}\maketitle}}

\begin{abstract}
The purpose of this paper is to prove that every finite set of points that can
be constructed in the Euclidean plane by using a compass and a ruler can also
be constructed by using unitary match-sticks in a non-simultaneous way and
following to a certain set of postulates. To prove this, we will deduce the
Euclidean axioms for our defined set of axioms.
\end{abstract}

\section{Introduction}
In 1939 T.R. Dawson proved (see \cite{Dawson}) that every point in the plane that can be constructed with a compass and a ruler can be constructed by laying match-sticks according to some rules. One of his most important postulates is, that two sticks may be laid simultaneously. Some similarly interesting constructions are done in \cite{Wells}.

 Our aim in the following is to define a set of postulates that are equivalent to the Euclidean postulate systems. These postulates will be similar to Dawson's, except that no two match-sticks can be used at the same time. Thus we call our constructions "non-simultaneous".

We will prove that
\begin{itemizeminus}
  \item we can extend every bounded line segment
  
  \item we can draw a line between any two given points
  
  \item we can construct right angles
  
  \item there exists a unique parallel line to a given line and through a given point
  
  \item we can describe circles geometrically. Since the constructed geometry is a point-geometry, it is sufficient to prove that we can construct the geometrical intersection between 2 given circles, respectively one circle and one line.
\end{itemizeminus}
The first point to prove is necessary since we only work on bounded match-sticks. This point will in fact occupy the major part of this work. Since the constructed geometry is included in the Euclidean space and
the definition of lines coincide, the fourth statement can be reduced to the constructibility of the only parallel line.

\section{The postulates}

The following are the postulates we will be working on for the rest of this paper.

\begin{enumerateroman}
  \item One may lay a match-stick through any two given points at distance
  equal or less than the unit, where one or even two of the points may be
  chosen as extremity of the stick segment. This procedure also permits us to
  determine whether the distance of two given points is strictly less, equal
  or strictly greater than the unit.
  
  \item One may choose an arbitrary point on a given stick segment, and it can
  be chosen such that it is none of the extremities.
  
  \item From a given point, one may use a match-stick as a unitary ruler. This
  ruler does not leave a trace, but can be used to determine a point with unit
  distance to the given point and relatively to a given set of match-sticks.
  No two match-sticks can be used as a compass at the same time.
\end{enumerateroman}

Notice that postulates i. and iii. are directly related to the non-simultaneousness of our construction constraints. Postulate ii. is a simple rule that seems necessary, but natural.

\section{Basic constructions}

We start by making some elementary constructions. These construction schemes will serve as building bricks for the more sophisticated constructions that follow. Most of these constructions directly use the proposed postulates. The first result will solve the problem that our postulates only construct bounded objects. We have to prove the following proposition:

\begin{proposition}
  Every bounded line can be extended to an arbitrary length.
\end{proposition}

\begin{proof}
  Let $A$, $B$ be the extremities of a given stick. Using postulate ii. we may
  choose a point $P$ in the interior of the stick. Using postulate i. we lay down a
  stick having $P$ as an extremity and passing through $B$, and its other
  extremity shall be denoted by $Q$. A third stick having $B$ as one extremity and
  passing through $Q$, is a translation of the first and has exactly one point in
  common with it. By reusing this algorithm, we can enlarge any given segment.
  Therefore, we may in the following draw lines in the classical way.
\end{proof}

Remark that this allows us to determine the intersection point of two given (i.e. determined by respective line segments) lines in a finite time. It does however not enable us to discover whether two lines are parallel. We will give another way for this later on. 

Now we construct perpendicular lines. Notice that the nature of constraints in non-simultaneousness renders unappliable most common constructions, based for example on the perpendicularity of diagonals in a rhombus.

\begin{proposition}
  For any line d and any $A \in d$, we can construct the line perpendicular to
  $d$ at $A$.
\end{proposition}

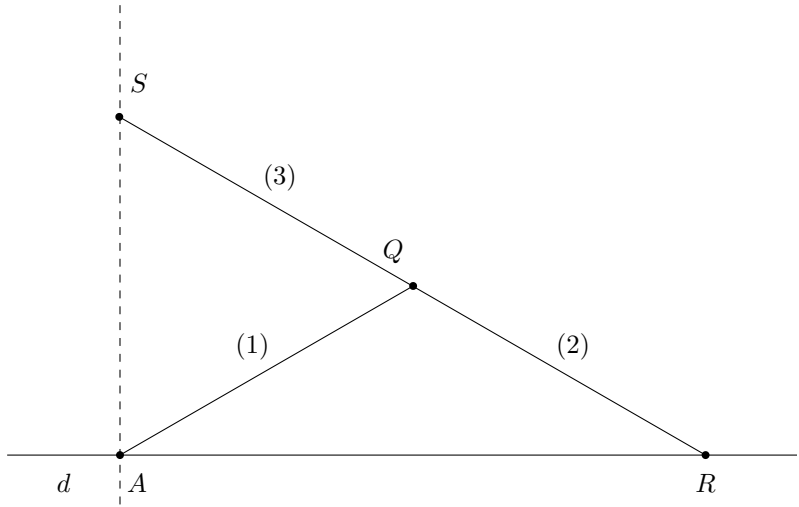
\begin{figure}[H]
\centering

\begin{tikzpicture}[x=1.5cm,y=1.5cm]
\clip(-3.5,-0.5) rectangle (4,4.5);

\draw (-3,0)-- (6,0);
\draw (-2.5,-0.25) node {$d$};
\fill(-2,0) circle (1.5pt);
\draw (-1.85,-0.25) node {$A$};
\draw[shift={(-2,0)},rotate=30] (0,0)--(3,0);
\draw[shift={(-2,0)},rotate=30] (1.5,0.25) node {(1)};
\draw[shift={(3.19,0)},rotate=150] (0,0)--(6,0);
\draw[shift={(3.19,0)},rotate=150] (1.5,-0.25) node {(2)};
\draw[shift={(3.19,0)},rotate=150] (4.5,-0.25) node {(3)};
\fill(3.19,0) circle (1.5pt);
\draw (3.19,-0.25) node {$R$};
\fill[shift={(-2,0)},rotate=30] (3,0) circle (1.5pt);
\draw[shift={(-2,0)},rotate=30] (3,0.35) node {$Q$};
\fill[shift={(3.19,0)},rotate=150] (6,0) circle (1.5pt);
\draw[shift={(3.19,0)},rotate=150] (6,-0.35) node {$S$};

\draw[dashed] (-2,-1)--(-2,4);

\end{tikzpicture}

\caption{Construction of a perpendicular line}
\end{figure}

\begin{proof}
  Choose a unit segment $(1)$ having $A$ as one of its extremities and not lying on $d$. Let us call $Q$ the second extremity of $(1)$. $(2)$ is laid such that
  its second extremity $R$ lies on $d$, but is different from $A$. $(3)$ is the
  extension of $(2)$ as shown on the picture. As we have $1 =
  \overline{{QA}} = \overline{{QR}} = \overline{{QS}}$ and $Q =
  {mid} (R, S)$, Thales' theorem implies that $\angle {SAR} =
  90^{\circ}$. An easy computation shows that this construction only works if the
  angle formed by the initial line and the first match-stick is less than $30^{\circ}$.
  Note that the maximal distance between $A$ and $S$ is less than the hypotenusis,
  which has length of two. Therefore, by doing this construction 3 times with different starting segements $(1)_1\neq (1)_2 \neq (1)_3$, the
  pigeonhole principle yields that at least two of the points $S_1$ , $S_2$ , $S_3$ are
  on the same side of $d$. The above geometrical argument shows that they
  both have to lie on the perpendicular line. As their maximal distance from $A$
  is 2, either two of the constructed points are closer than one unit from each other, or one of them has
  a distance from $A$ which is less than one unit. In both cases, Postulate i. and
  Proposition 1. allow us to construct the perpendicular line to $A$.
\end{proof}

The next step will enable us to construct coordinate grids. We will call coordinate grid the covering of the plane with
unit squares that are all parallel to each other. Since we work on a finite number of steps, the aim will be to define a grid that contains each point after a finite number of steps.

\begin{proposition}
  For any two points $A, B$ in the plane, there exists a construction sequence
  of a coordinate grid such that $A$ and $B$ are contained in the grid after a
  finite number of construction steps.
\end{proposition}

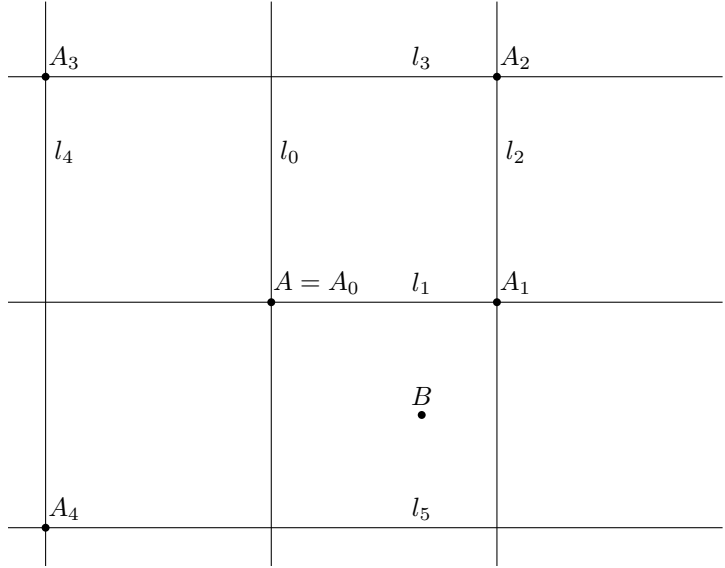
\begin{figure}[H]
\centering

\begin{tikzpicture}[x=1.0cm,y=1.0cm]
\clip(-3.5,-3.5) rectangle (6.5,5.5);

\fill(0,0) circle (1.5pt);
\draw(0.6,0.25) node {$A=A_0$};

\fill(2,-1.5) circle (1.5pt);
\draw(2,-1.25) node {$B$};

\draw(0,-4)--(0,4);
\draw(0.25,2) node {$l_0$};

\fill(3,0) circle (1.5pt);
\draw(3.25,0.25) node {$A_1$};
\fill(3,3) circle (1.5pt);
\draw (-4,0)-- (6,0);
\draw (2,0.25) node {$l_1$};

\draw(3.25,3.25) node {$A_2$};
\fill(-3,3) circle (1.5pt);
\draw (3,-4)--(3,4);
\draw (3.25,2) node {$l_2$};

\draw(-2.75,3.25) node {$A_3$};
\fill(-3,-3) circle (1.5pt);
\draw(-4,3)--(6,3);
\draw(2,3.25) node {$l_3$};

\draw(-2.75,-2.75) node {$A_4$};
\draw(-3,4)--(-3,-4);
\draw(-2.75,2) node {$l_4$};

\draw(-4,-3)--(6,-3);
\draw(2,-2.75) node {$l_5$};

\end{tikzpicture}

\caption{Construction of a coordinate grid}
\end{figure}

\begin{proof}
  Let $l_0$ be a line containing $A=A_0$ and $l_1$ be its perpendicular line passing through $A_0$. Chose $A_1$ on $l_1$ having unit distance from $A_0$. Then construct the perpendicular $l_2$ to $l_1$ and passing through $A_2$. Chose a point $A_2$ in a way similar to $A_1$ and construct $l_3$ the perpendicular to $l_2$ through $A_2$. $A_3$ is chosen on $l_3$ as having distance 2 from $A_2$ and such that $[A_2,A_3]$ intersects $l_0$ (which is finitely constructible). At each recursive step, we chose $A_i$ on $l_i$ having distance $\lceil\frac{i}{2}\rceil$ from $A_{i-1}$ and such that $[A_{i-1},A_i]$ intersects $l_{i-3}$. This then covers the plan "spirally". Since the distance between $A$
  and $B$ is finite, both will be in this grid sequence after a finite number
  of steps.
\end{proof}

\section{Combined constructions}

We now have the means to make more general constructions. We start by constructing the perpendicular bisector and specific parallel lines.

\begin{proposition}
  If $\left[ A, B \right]$ is a line segment, we can construct its
  perpendicular bisector.
\end{proposition}

\begin{figure}[H]
\centering

\begin{tikzpicture}[x=1.5cm,y=1.5cm]
\clip(-3,-1.5) rectangle (6,4.25);

\draw(-4,0)--(6,0);
\fill(-2,0) circle (1.5pt);
\draw(-1.75,0.25) node {$A$};
\fill(5,0) circle (1.5pt);
\draw(5.25,0.25) node {$B$};

\draw(-2,-3)--(-2,3);
\draw(5,-3)--(5,3);
\fill(-2,2) circle (1.5pt);
\draw(-2.5,2.25) node {$R=R_0$};
\fill(-2,-1) circle (1.5pt);
\draw(-2.5,-0.75) node {$S=S_0$};
\fill(5,-1) circle (1.5pt);
\draw(5.5,2.25) node {$T=T_0$};
\fill(5,2) circle (1.5pt);
\draw(5.5,-0.75) node {$U=U_0$};

\draw(-2,2)--(0.24,0);
\fill(0.24,0) circle (1pt);
\draw(0.24,0.25) node {$R_1$};
\draw(0.24,0)--(2.48,2);
\fill(2.48,2) circle (1pt);
\draw(2.48,2.25) node {$R_2$};

\fill(2.76,0) circle (1pt);
\draw(2.76,0.25) node {$T_1$};
\draw(5,2)--(2.76,0);
\fill(0.52,2) circle (1pt);
\draw(0.52,2.25) node {$T_2$};
\draw(2.76,0)--(0.52,2);

\fill(0.83,0) circle (1pt);
\draw(0.83,0.25) node {$S_1$};
\draw(-2,-1)--(0.83,0);
\fill(3.66,-1) circle (1pt);
\draw(3.66,-1.25) node {$S_2$};
\draw(0.83,0)--(3.66,-1);

\fill(2.17,0) circle (1pt);
\draw(2.17,0.25) node {$U_1$};
\draw(5,-1)--(2.17,0);
\fill(-0.66,-1) circle (1pt);
\draw(-0.66,-1.25) node {$U_2$};
\draw(2.17,0)--(-0.66,-1);

\draw(-2,2)--(5,2);
\draw(-2,-1)--(5,-1);

\fill(1.5,1.11) circle (1.5pt);
\draw(1.75,1.46) node {$P$};
\fill(1.5,-0.26) circle (1.5pt);
\draw(1.75,-0.55) node {$Q$};

\fill(1.5,0) circle (1.5pt);
\draw(1.75,0.25) node {$C$};

\draw(1.5,-3)--(1.5,3);

\end{tikzpicture}

\caption{Construction of a perpendicular bisector}
\end{figure}
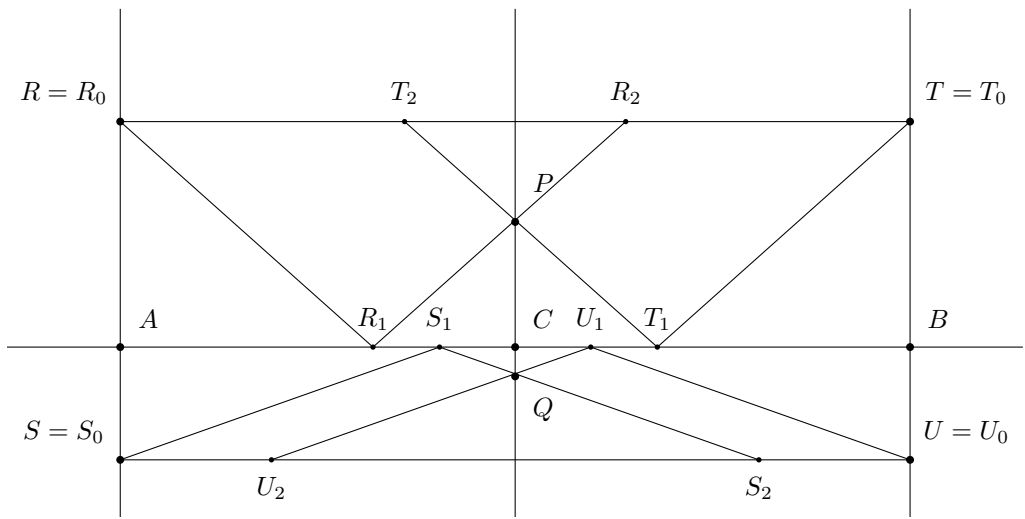

\begin{proof}
  Using Proposition 2, we construct the perpendicular lines to $AB$ at $A$ and $B$
  respectively. Using Postulate iii., we can choose a point $R=R_0$ on the
  perpendicular to $A$ with distance less than a unit from $AB$. A second
  point $S=S_0$ is constructed on the other side of $AB$ and at a unit distance from the
  first point. Now we know that perpendiculars at these points are both
  parallel to $AB$ and have unit distance. We then construct $T$ and $U$ the intersection points of these perpendiculars with the perpendicular through $B$ respectively. By using the process shown on the
  picture, we iteratively construct $R_i,T_i,S_i,U_i$ by laying down unit sticks from $R_{i-1},T_{i-1},S_{i-1}$ and $U_{i-1}$ respectively. We stop these two processes when $[R_{i-1},R_i]$ and $[T_{i-1},T_i]$ intersect, respectively when $[S_{i-1},S_i]$ and $[U_{i-1},U_i]$ intersect. This process then yields respective intersection points $P$ and $Q$ at a distance less than one.
  Therefore, we can draw the line through those two points. A basic
  geometrical argument proves that $P Q$ is the perpendicular bisector of $[A,B]$. As a corollary, we deduce that its intersection point with $A B$,
  called $C$, is the midpoint of the segment.
\end{proof}

Since the construction of parallel lines is closely related to the construction of perpendicular lines, the following result follows naturally:

\begin{proposition}
  Given a line $d$ and a point $A$ with distance from $d$ less than one unit, we
  can construct the perpendicular and the parallel lines to $d$ passing through $A$.
\end{proposition}

\begin{proof}
  Determine the two points $P$ and $Q$ with unit distance from $A$ with $d$. Using
  Proposition 4 and the fact that $\vartriangle A P Q$ is isosceles, we can
  construct the perpendicular bisector of $[P Q]$, which passes through $A$. The
  perpendicular to this line at $A$ is parallel to $d$, which concludes the
  proof.
\end{proof}

This result now gives us the announced method of determining whether two lines are parallel. It is enough to consider the parallel through one of the lines that passes through a point of the other line and check if they both are coincident.

\section{Advanced constructions}

We will now verify the validity of the remaining Euclidean postulates. We will prove the uniqueness of the line passing through 2 given points, as well as the pointwise constructibility of circles with given incenter and radius.

Most of the previous constructions have dealt with bounded figures. We now combine those constructions with the unit grid to extend our geometry to non-bounded alignment, parallelism and perpendicular lines.

\begin{proposition}
  If $A$ and $B$ are two distinct points, we can construct their commoun line. If $l$ and $p$ are given lines, we can also construct the parallel to $l$ which passes through $B$, as well as the perpendicular to $p$ which passes through $B$.
\end{proposition}

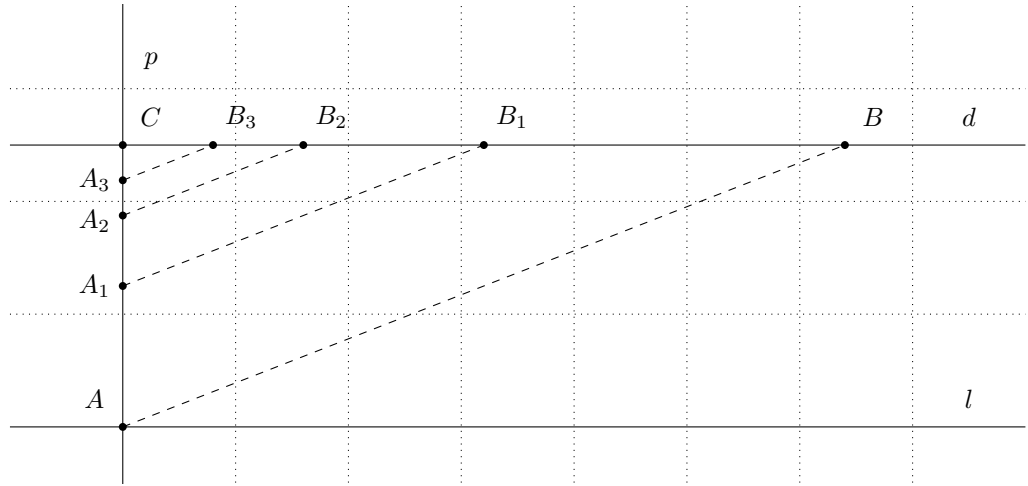
\begin{figure}[H]
\centering

\begin{tikzpicture}[x=1.5cm,y=1.5cm]
\clip(-4.5,-0.5) rectangle (6.5,3.75);

\fill(-3,0) circle (1.5pt);
\draw(-3.25,0.25) node {$A$};
\fill(3.4,2.5) circle (1.5pt);
\draw(3.65,2.75) node {$B$};
\draw[dashed](-3,0)--(3.4,2.5);
\draw(-4,0)--(5,0);
\draw[dotted](-4,1)--(5,1);
\draw[dotted](-4,2)--(5,2);
\draw[dotted](-4,3)--(5,3);
\draw(-3,-1)--(-3,4);
\draw[dotted](-2,-1)--(-2,4);
\draw[dotted](-1,-1)--(-1,4);
\draw[dotted](0,-1)--(0,4);
\draw[dotted](1,-1)--(1,4);
\draw[dotted](2,-1)--(2,4);
\draw[dotted](3,-1)--(3,4);
\draw[dotted](4,-1)--(4,4);

\draw(4.5,0.25) node {$l$};
\draw(-2.75,3.25) node {$p$};

\draw(-4,2.5)--(5,2.5);
\draw(4.5,2.75) node {$d$};
\fill(-3,2.5) circle (1.5pt);
\draw(-2.75,2.75) node {$C$};

\fill(-3,1.25) circle (1.5pt);
\draw(-3.25,1.25) node {$A_1$};
\fill(0.2,2.5) circle (1.5pt);
\draw(0.45,2.75) node {$B_1$};
\draw[dashed](-3,1.25)--(0.2,2.5);
\fill(-3,1.875) circle (1.5pt);
\draw(-3.25,1.825) node {$A_2$};
\fill(-1.4,2.5) circle (1.5pt);
\draw(-1.15,2.75) node {$B_2$};
\draw[dashed](-3,1.875)--(-1.4,2.5);
\fill(-3,2.1875) circle (1.5pt);
\draw(-3.25,2.1875) node {$A_3$};
\fill(-2.2,2.5) circle (1.5pt);
\draw[dashed](-3,2.1875)--(-2.2,2.5);
\draw(-1.95,2.75) node {$B_3$};
\end{tikzpicture}

\caption{Construction of the line between two points}
\end{figure}

\begin{proof}
  We start by constructing a line $l$ passing through $A$ and let $p$ be the
  perpendicular to $l$ passing through $A$. If $B$ lies on any of those lines,
  the proof is finished. Otherwise, we construct, using Proposition 3, a coordinate unit grid square having $(A, l, p)$ as determing lines and being sufficiently large to contain $B$. Let $l_i$ , $l_{i + 1}$ and $p_j$ , $p_{j + 1}$ be the lines defining the unit square that contains $B$. As the distance
  of $B$ from any four of those lines is less than 1 unit, we may use Proposition 5 and
  draw a line $d$ perpendicular with $p_j$ (and hence to $p$) passing through $B$. Similarly, we can obtain the line perpendicular with $d$ through $B$. As $d$ is also
  perpendicular to $p$, let $C$ be its intersection with $l$. Construct the
  midpoints $A_1$ of $\left[ A C \right]$ and $B_1$ of $\left[ B C \right]$. Using the intercept theorem, we know that the segment drawn between those two points should be
  parallel to $\left[ A B \right]$ and have half of its length. Reusing this
  algorithm, we can go on by constructing the midpoints of the new segments
  and so on. After a finite number of steps, the distance between the two
  actual midpoints is less than a unit. Therefore, we can draw the line
  between those two points. Using again the previous construction, we can now
  construct the parallel to this line and through $B$, which passes also
  through $A$. 
\end{proof}

We will now move on to the circular intersections. Note that there are two
different definitions of a circle. On one hand, a circle can be given as the
set composed of its center and one of its points. On the other hand, it can
also be given as the set composed of the center and a segment defining its
radius and not necessarily having the center as one of its extremities. The
construction of a parallelogram enables us to translate the given segment to a
suitable position. Therefore, whenever in the following, we are referring to a
circle, we consider that we know its center and one of its points. We first start by intersecting a line and a circle.

\begin{proposition}
  Let $A B$ be a line and $\Gamma$ be a circle with centre $O$ and $S \in
  \Gamma$. Then we can construct the intersection of $(AB)$ and $\Gamma$,
  respectively determine whether this intersection is empty.
\end{proposition}

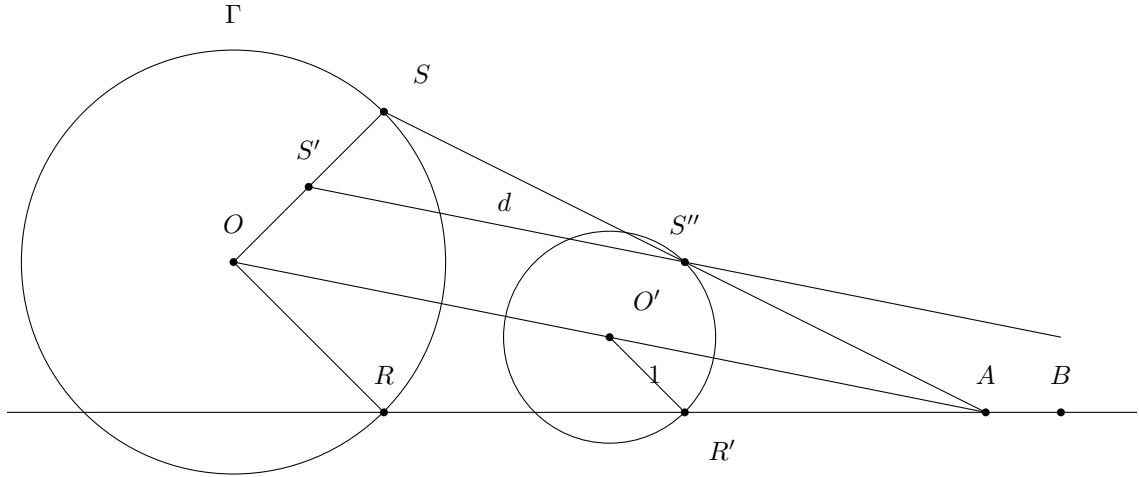
\begin{figure}[H]
\centering

\begin{tikzpicture}[x=2cm,y=2cm]
\clip(-1.5,-0.5) rectangle (6,4.5);

\draw(0,1) circle (1.41);
\fill(0,1) circle (1.5pt); 
\draw(0,1.25) node {$O$};
\draw(0,2.65) node {$\Gamma$};
\draw(-3,0)--(7,0);
\fill(5,0) circle (1.5pt);
\draw(5,0.25) node {$A$};
\fill(5.5,0) circle (1.5pt);
\draw(5.5,0.25) node {$B$};
\fill(2.5,0.5) circle(1.5pt);
\draw(2.75,0.75) node {$O'$};
\fill(1,2) circle(1.5pt);
\draw(1.25,2.25) node {$S$};
\fill(0.5,1.5) circle (1.5pt);
\draw(0.5,1.75) node {$S'$};
\fill(3,1) circle (1.5pt);
\draw(3,1.25) node {$S''$};
\draw(0,1)--(5,0);
\draw(0,1)--(1,2);
\draw(1,2)--(5,0);
\draw(2.5,0.5) circle (0.705);
\fill(3,0) circle (1.5pt);
\draw(3.25,-0.25) node {$R'$};
\draw(3,0)--(2.5,0.5);
\draw(2.8,0.25) node {$1$};
\draw(0.5,1.5)--(5.5,0.5);
\draw(1.8,1.4) node {$d$};
\fill(1,0) circle (1.5pt);
\draw(1,0.25) node {$R$};
\draw(0,1)--(1,0);

\end{tikzpicture}
\caption{Intersection of a circle and a line}
\end{figure}

\begin{proof}
  If $d \left( O, S \right) = 1$, the (possibly empty) intersection is
  obtained by the simple use of postulate iii. Otherwise, chose $S'$ on
  ${(OS)}$ such that $\overline{O S'} = 1$. Let $d$ be the parallel line to
  $A O$ that goes through $S'$. Let $S''$ be the intersection point of $(AS)$
  and $d$. Let $O'$ be the intersection point of $AO$ with the parallel to $O S$ that passes
  through $S''$. There are 0, 1 or 2 point on $AB$ with unit
  distance from $O'$, determining the number of intersection points between $AB$
  and $\Gamma$. If such a point $R'$ at unit distance exists, then the
  intersection point of $A B$ and the parallel to $O' R$' is both on $AB$ and
  on $\Gamma$, which gives the announced construction.
\end{proof}

Now that we can intersect lines and circles, we will intersect two given circles. The idea will be to transform this intersection problem into an intersection problem of the previous type.

\begin{proposition}
  Let $\Gamma_1$ and $\Gamma_2$ be two circles with centres $O_1$ and $O_2$
  and radii $r_1$ and $r_2$. We can determine whether they are intersective.
  If they are, we can find their intersection points.
\end{proposition}

\begin{figure}[H]
\centering

\begin{tikzpicture}[x=1.0cm,y=1.0cm]
\clip(-3.5,-5.5) rectangle (9.5,6.5);

\draw(0,-2) circle (3cm);
\fill (0,-2) circle (1.5pt) coordinate (O1);
\draw (0,-2.25) node {$O_1$};
\draw (0,1.25) node {$\Gamma_1$};

\draw (-1.5,-2.25) node {$r_1$};
\draw (0,-2)-- (-3,-2);

\draw(4,-1) circle (2cm);
\fill (4,-1) circle (1.5pt) coordinate (O2);
\draw (4,-1.25) node {$O_2$};
\draw (4,1.25) node {$\Gamma_2$};

\draw[shift={(4cm,-1cm)},rotate=35] (0,0)--(2,0);
\draw[shift={(4cm,-1cm)},rotate=35] (1,-0.25) node {$r_2$};

\draw (4,3) circle (3cm);
\fill(4,3) circle (1.5pt) coordinate(O3);
\draw(4,3.25) node {$O_3$};
\draw(4,6.25) node {$\Gamma_3$};

\draw[dotted] (O1)--(O3);
\draw[dotted] (O1)--(O2);

\draw(4,3)--(1,3);
\draw(2.5,3.25) node {$r_1$};

\fill[shift={(4cm,-1cm)},rotate=44](2,0) circle(1.5pt) coordinate (X);
\draw[shift={(4cm,-1cm)},rotate=44](2.25,0.25) node {$X$};

\fill[shift={(4cm,-1cm)},rotate=136](2,0) circle(1.5pt) coordinate (Y);
\draw[shift={(4cm,-1cm)},rotate=136](1.75,0) node {$Y$};

\draw[dotted](X) circle (3cm);

\draw[dotted](4,-5)--(4,6.5);

\draw[shift={(X)}](-10,0)--(5,0);
\draw[shift={(X)}](2,0.25) node {$a_1$};

\draw[shift={(2,0.5)}](-5,4)--(5,-4);
\draw[shift={(2,0.5)}](-4,3.5) node {$a_2$};

\fill[shift={(X)}](-3.3,0) circle(1.5pt) coordinate (P);
\draw[shift={(P)}](0.1,0.25) node {$P$};

\draw[shift={(P)}](-1,4)--(1.5,-6);
\draw[shift={(P)}](1.4,-5) node {$a_3$};

\end{tikzpicture}

\caption{Intersection of two Circles}
\end{figure}
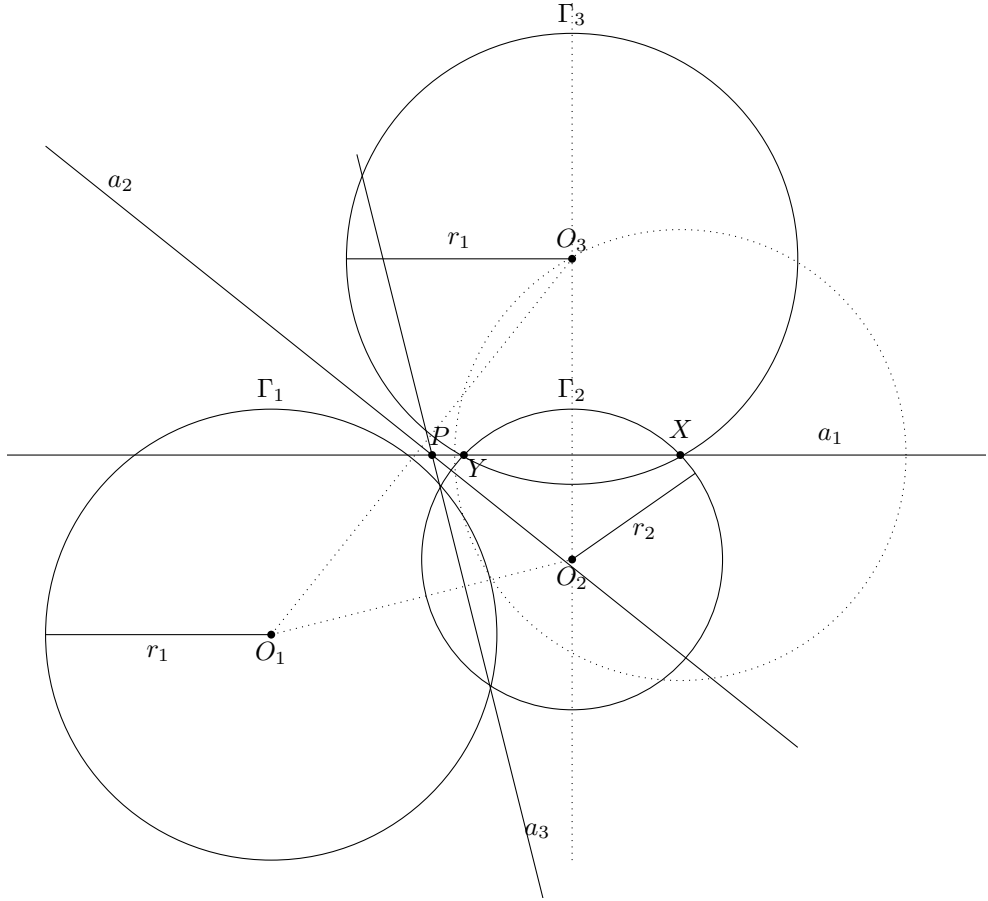

\begin{proof}
  Let $X$ and $ Y$ be two points on $\Gamma_2$ such that $(O_1 O_2)$ and $(XY)$ are not perpendicular. This is a mild assumption since we can easily construct the antipodal and the respective $90^{\circ}$-rotations of a given point on the circle, producing at least one non-perpendicular line. Using proposition 7, we can construct the
  intersection point of the circle with centre $X$ and radius $r_1$ and the
  perpendicular bisector of segment $\left[ X Y \right]$. Call this point
  $O_3$. Now, consider the circle $\Gamma_3$ with centre $O_3$ and radius
  $r_1$. The radical axis $a_1$ of $\Gamma_2$ and $\Gamma_3$ is $(X Y)$, as
  $\Gamma_2 \cap \Gamma_3 =\{X, Y\}$. The radical axis $a_2$ of $\Gamma_1$ and
  $\Gamma_3$ is the perpendicular bisector of the segment $\left[ O_1 O_3
  \right]$. This can be proved by symmetry, as $\Gamma_1$ and $\Gamma_3$ have
  same radius. By using the radical axis theorem, there exists a power point $P$ which is the intersection of $a_1$ and $a_2$ (which is non-empty by construction).
  Therefore, the perpendicular $a_3$ line to $(O_1 O_2)$ which passes through $P$ has to be the radical axis of $\Gamma_1$ and $\Gamma_2$. The
  intersection of this line with $\Gamma_1$ is equal to the (possibly empty)
  intersection of the two circles.
\end{proof}

\section{Conclusion}

We have proved the equivalence between the Euclidean postulate system and our new bounded system. One main advantage of this method is that it only uses bounded constructions but still obtainis the same results. Remark that it is easy to generalize our construction to the case where the unit segment and the ruler have length ratios that are not $1$, but any rational number. If however this ratio is an irrational number, it might be interesting to construct a geometry satisfying the postulates, but different from the Euclidean geometry.

\section*{Acknowledgements}
The authors like to thank the organizers of the First Internationational Summer School for Students in 2011. Without the inspiration the authors were given when they met at this occasion, this paper probably never would have been written. Thanks also go to Christophe Ley for his constructive remarks, corrections and suggestions. At the end, they would like to thank all their acquaintances for providing them with match-sticks and toothpicks whenever there was need for checking the correctness of their constructions.

\end{document}